\documentclass[10pt]{amsart}
\usepackage[margin=1.2in,marginparsep=0.1in,marginparwidth=1in]{geometry}
\usepackage{amssymb,amsmath,amsthm,amstext,amscd,latexsym,graphics,graphicx,bbm,caption}
\usepackage[usenames,dvipsnames,svgnames,table]{xcolor}
\usepackage[plainpages=false,colorlinks=true, pagebackref]{hyperref}
\usepackage{tikz}
\usetikzlibrary{positioning}

\tikzset{>=stealth}
\hypersetup{citecolor=Sepia,linkcolor=blue, urlcolor=blue}

\usepackage{cite}   

\makeatletter
\def\@tocline#1#2#3#4#5#6#7{\relax
  \ifnum #1>\c@tocdepth 
  \else
    \par \addpenalty\@secpenalty\addvspace{#2}%
    \begingroup \hyphenpenalty\@M
    \@ifempty{#4}{%
      \@tempdima\csname r@tocindent\number#1\endcsname\relax
    }{%
      \@tempdima#4\relax
    }%
    \parindent\z@ \leftskip#3\relax \advance\leftskip\@tempdima\relax
    \rightskip\@pnumwidth plus4em \parfillskip-\@pnumwidth
    #5\leavevmode\hskip-\@tempdima
      \ifcase #1
       \or\or \hskip 2em \or \hskip 2em \else \hskip 3em \fi%
      #6\nobreak\relax
    \dotfill\hbox to\@pnumwidth{\@tocpagenum{#7}}\par
    \nobreak
    \endgroup
  \fi}
\makeatother

\newtheorem{intro-thm}{Theorem}[]
\theoremstyle{plain}
\newtheorem{thm}{Theorem}[section]
\newtheorem{theorem}[thm]{Theorem}
\newtheorem{question}[thm]{Question}
\newtheorem{lemma}[thm]{Lemma}
\newtheorem{corollary}[thm]{Corollary}

\theoremstyle{definition}
\newtheorem{remark}[thm]{Remark}
\newtheorem{point}[thm]{}

\newcommand{\ilim}{\mathop{\varprojlim}\limits} 

\newcommand{\inj}{\hookrightarrow}

\newcommand{\Union}{\bigcup}
\newcommand{\intersection}{\cap}

\newcommand{\Ker}{{\rm Ker  }}

\renewcommand{\tilde}{\widetilde}

\newcommand{\sR}{{\mathcal R}}

\newcommand{\A}{{\mathbb A}}

\newcommand{\N}{{\mathbb N}}

\newcommand{\Z}{{\mathbb Z}}

\let\syn\mathsf

\newcommand{\scr}{\scriptscriptstyle}

\newcommand{\ca}{\frac{A}{[A,A]}}

\input{xy}
\xyoption{all}

\begin{document}

\title[Witt vectors of non-commutative rings]
{On the comparison of two constructions of Witt vectors of non-commutative rings} 

\author[A. Hogadi]{Amit Hogadi} \address{Indian Institute of
Science, Education and Research (IISER),  Homi Bhabha Road, Pashan,
Pune - 411008, India} \email{amit@iiserpune.ac.in} 

\author[S. Pisolkar]{Supriya Pisolkar} \address{ Indian Institute of
Science, Education and Research (IISER),  Homi Bhabha Road, Pashan,
Pune - 411008, India} \email{supriya@iiserpune.ac.in} 

\date{}
\begin{abstract} Let $A$ be any associative ring, possibly non-commutative and let $p$ be a prime number. Let $E(A)$ be the ring of $p$-typical Witt vectors as constructed by Cuntz and Deninger in \cite{dc2} and $W(A)$ be that constructed by Hesselholt in \cite{h2}. The goal of this paper is to answer the following question by Hesselholt : Is $HH_0(E(A)) \cong W(A)?$ We show that in the case $p=2$, there is no such isomorphism possible if one insists that it be compatible with the Verscheibung operator and the Teichm\"uller map.
\end{abstract}

\maketitle

\section{Introduction}

Let $A$ be a associative unital ring and $p$ be a prime number. When $A$ is commutative, the classical construction of $p$-typical Witt vectors gives us a topological ring $W(A)$, equipped with a Verscheibung operator 
$$ W(A) \xrightarrow{V} W(A)$$ 
and a Teichm\"uller map, which we denote by 
$$ A \xrightarrow { \langle \ \rangle } W(A).$$

In this paper we consider two generalizations of this construction to the non-commutative case. 
One of them is a construction of an abelian group $W(A)$ given by Hesselholt in \cite{h1} (see also \cite{h2}). The other is a construction of a ring $E(A)$ by Cuntz and Deninger, given in \cite{dc2}.  Just as in the commutative case, $E(A)$ and $W(A)$ are topological rings and 
are equipped with the Verscheibung operator and the Teichm\"uller map. Moreover, both $W(A)$ and $E(A)$ are isomorphic to the classical construction of Witt vectors when $A$ is commutative.  Let $HH_0(E(A)):= E(A)/\overline{[E(A),E(A)]}$. The goal of this paper is to answer the following question of Hesselholt. 
\begin{question}\label{hq}
Is  $W(A)$ isomorphic to  $HH_0(E(A))$? 
\end{question}

We note that $HH_0(E(A))$ inherits the Verschiebung operator $V$ and the Teichm\"uller map $\langle \ \rangle $, from $E(A)$.  In this paper we only consider maps from $W(A)$ to $HH_0(E(A))$ which are compatible with $V$ and $\langle \ \rangle $. \\

\noindent The following is one of the main results of this paper. 
\begin{theorem}\label{thm1}
Let $A={\Z}\{X,Y\}$ and $p=2$.  Then
\begin{enumerate}
\item[(i)] $W(A)$ is topologically generated by $\{ V^n(\langle a \rangle) \ | \ n\in \N_0,  a\in A \}$.
\item[(ii)] $HH_0(E(A))$ is not topologically generated by $\{ V^n(\langle a \rangle) \ | \ n\in \N_0,  a\in A \}$.
\item[(iii)] there is no continuous surjective map from $W(A)\to HH_0(E(A))$ which commutes with $V$ and is compatible with $\langle \  \rangle$. 
\end{enumerate}
\end{theorem} 
  
One can thus slightly modify Question \ref{hq} and ask for existence of a map (not necessarily surjective) from $W(A) \to HH_0(E(A))$ compatible with the Verscheibung operator and the Teichm\"uller map. The following theorem, shows that even this is not possible, at least in the case $p=2$. 

\begin{theorem}\label{thm2} Let $A:=\Z\{X,Y\}$ and $p=2$. Then there is no continuous group homomorphism from $W(A) \to HH_0(E(A))$ which is compatible with $V$ and $\langle \ \rangle$. 
\end{theorem}
 
We believe that Theorems \ref{thm1} and \ref{thm2} should hold for all primes $p$. \\

\noindent {\bf Acknowledgement}: We are highly indebted to Lars Hesselholt for drawing our attention to this question and for his encouraging comments on the first draft of this paper. We also thank J. Cuntz and C. Deninger for their comments on the first draft 
of this paper.
 
\section{Preliminaries}

Let $A$ be a unital associative ring, not necessarily commutative. Fix a prime number $p$. Let 
$\N_0:= \N\Union \{0\}$. We briefly recall a few facts about the constructions of Witt ring by Hesselholt and by Cuntz and Deninger.

\begin{point} {\bf Hesselholt's construction of $W(A)$}: \label{construction-H}
Let $ A$ be a unital associative ring. Consider the map (called as ghost map) 
$$ \omega: A^{\N_0} \to \Big( \ca \Big)^{\N_0} $$  
$$ \omega(a_0,a_1,a_2,...) := \big(\omega_0(a_0), \omega_1(a_0,a_1), \omega_2(a_0,a_1,a_2),...\big)$$
where $\omega_i$'s are ghost polynomials defined by 
$$\omega_i(a_0,...,a_i) := a_0^{p^i}+pa_1^{p^{i-1}}+p^2a_2^{p^{i-2}}+\cdots+p^i a_i.$$
$\omega$ is merely a map of sets and not a homomorphism of groups. 
For every integer $n\in \N_0$, we also have truncated versions of the above map (denoted again by $\omega$)
$$ \omega: A^n \to \Big( \ca \Big)^{n}$$
$$ \omega(a_0,...,a_{n-1}) := \big(\omega_0(a_0), \omega_1(a_0,a_1),..., \omega_{n-1}(a_0,a_1,a_2,...a_{n-1})\big)$$
\noindent Hesselholt then inductively defines groups $W_n(A)$ (see \cite{h2}) such that 
the map $\omega$ factor through 
$$A^n \xrightarrow{q_n} W_n(A) \xrightarrow{\overline{\omega}} \Big(\ca\Big)^n$$ 
and the following are satisfied 
\begin{enumerate} 
\item $W_1(A) = \ca$. 
\item $q_n$ is surjective map of sets. 
\item $\overline{\omega}$ is an additive homomorphism and is injective if $\ca$ is $p$-torsion free.
\end{enumerate}
\vspace{2mm}

\noindent Define $W(A) := \ilim_n W_n(A)$. Clearly one also has a factorization of $A^{\N_0}\xrightarrow{\omega} \big(\ca\big)^{\N_0}$ as 
$$ A^{\N_0}\xrightarrow{q} W(A) \xrightarrow{\overline{\omega}} \Big( \ca \Big)^{\N_0}$$
where $q$ is always surjective and where $\overline{\omega}$ is injective if $\ca$ has no $p$-torsion. Thus every element of $W(A)$ is of the form $q(a_0,a_1,...)$ for some (not necessarily unique) $(a_0,a_1,...)\in \A^{\N_0}$. We will often abuse notation to denote an element of $W(A)$ simply by a tuple $(a_0,a_1,....)$  instead of the cumbersome $q(a_0,a_1,...)$ while taking care to handle well definedness issues should they arise. \\

\noindent We  have the Verschiebung operator  
$$V: W(A) \to W(A) $$ 
and the Teichm\"uller map 
$$ \langle \ \rangle : A \longrightarrow W(A) $$
which satisfy

$$V (a_0,a_1,\cdots) = (0,a_0,a_1,\cdots)$$
 and
$$\langle a\rangle= ( a,0,0,\ldots)$$

\noindent One can show that $V$ and $\langle \ \rangle$ are well defined and that $V$ is a additive homomorphism. Similarly for $n\in \N_0$, we have truncated versions (denoted by the same notation) $W_n(A) \xrightarrow{V} W_n(A) $ and $A \xrightarrow{\langle \ \rangle} W_n(A)$, satisfying

$$ V(a_0,...,a_{n-1}) = (0,a_0,a_1,...,a_{n-2}).$$
and 
$$ \langle a \rangle = (a,0,0,...,0).$$ 

\end{point}

\begin{point} {\bf Cuntz and Deninger's construction of $E(A)$}: \label{construction-CD}
For any associative ring $R$ 
\begin{enumerate}
\item[(i)]  Let $V: R^{\N_0} \to R^{\N_0}$ be the map defined by  
$V(a_0,a_1,... ) := p(0,a_0,a_1,....)$.
\item[(ii)] For an element $a\in R$, define $ \langle a \rangle \in R^{\N_0}$ by 
$ \langle a \rangle := (a,a^p,a^{p^2},...)$.
\item[(iii)] Let $X(R)\subset R^{\N_0} $ be the closed subgroup generated by 
$$ \Big\{ V^m(\langle a_1 \rangle \cdots  \langle a_r \rangle ) \ | \ m \in \N_0, r\in \N, \ a_i \in R  \ \forall \ i\Big\}.$$
\noindent Similarly, if $I\subset R$ is an ideal, we let $X(I)$ denote the closed subgroup generated by 
$$ \Big\{ V^m(\langle a_1 \rangle \cdots  \langle a_r \rangle ) \ | \ m  \in \N_0, r\in \N, \ a_i \in I \ \forall \ i \Big\}.$$
\end{enumerate}
For $n\in \N_0$, we also have the following truncated versions of the definitions. 
\begin{enumerate}
\item[(i)]  Let $V: \prod_{i=0}^n R \to \prod_{i=0}^n R$ be the map defined by  
$V(a_0,a_1,...,a_n ) := p(0,a_0,a_1,...,a_{n-1})$.
\item[(ii)] For an element $a\in R$, define $ \langle a \rangle \in \prod_{i=0}^n R$ by 
$ \langle a \rangle := (a,a^p,a^{p^2},...,a^{p^n})$.
\item[(iii)] Let $X_n(R)\subset \prod_{i=0}^n R $ be the subgroup generated by 
$$ \Big\{ V^m(\langle a_1 \rangle \cdots  \langle a_r \rangle ) \ | \ m \in \N_0, r\in \N,  \ a_i \in R \ \forall \ i\Big\}.$$
\end{enumerate}

Note that the above definitions depend on the chosen prime number $p$. Now let $A$ be any associative ring. Let $\Z A$ be the monoid algebra of the multiplicative monoid underlying $A$. Thus the elements of $\Z A$ are formal sums of the form $ \sum_{r\in \Z A} n_r [r] \ \ \ \text{with almost all }n_r=0$. We have a natural epimorphism of rings from $\Z A \to A$ and we let $I$ denote its kernel. One now defines 
$$ E(A) := \frac{X(\Z A)}{X(I)}  \ \ \ \text{and} \ \ \ E_n(A):= \frac{X_n(\Z A)}{X_n(I)} $$ 
We have a set map, called the Teichm\"uller map
$$ \langle \ \rangle : A \to E(A) $$
$$ \langle a \rangle := ([a],[a]^p,[a]^{p^2},...) \ \ \syn{mod} \ X(I).$$
It is elementary to check that 
\begin{enumerate}
\item[(i)] $X(\Z A)$ is a subring of $(\Z A)^{\N_0}$ and $X(I) \subset X(\Z A)$ is a two sided ideal. Thus $E(A)$ has the structure of an associtive ring. Similarly $E_n(A)$ has a ring structure $ \forall \ n\in \N_0$. 
\item[(ii)] $V:({ \Z A})^{\N_0} \to (\Z A)^{\N_0}$ induces maps (denoted by the same notation for simplicity) called Verscheibung operators 
$$  E(A) \xrightarrow{V} E(A) \ \ , \ \  E_n(A) \xrightarrow{V} E_n(A). $$
\item[(iii)] The morphism $E(A)\to E_n(A)$ is a continuous epimorphism of rings, where target has discrete topology and $E(A)$ has topology inherited from the product topology on $(\Z A)^{\N_0}$. Moreover this map commutes with $V$.  
\end{enumerate}
Recall that for any associative ring $R$, $HH_0(R):= R/[R,R]$. However, as $E(A)$ has a topology, we use the notation $HH_0(E(A))$ to denote $E(A)/\overline{[E(A),E(A)]}$ where $\overline{[E(A),E(A)]}$ is the closure of $[E(A),E(A)]$. The Verscheibung operator and the Teichm\"uller map induce (see \eqref{vinduces})
$$ V: HH_0(E(A)) \to HH_0(E(A))$$ 
and $$ \langle \ \rangle : A \to HH_0(E(A)).$$
\end{point}

\begin{point}{\bf Comparison of $E(A)$ and  $X(A)$}: \label{eaxa} In the case when $A$ is commutative and has no $p$-torsion, $E(A)$ is isomorphic to $X(A)$. In the non-commutative case we do not know a sufficient condition for  $E(A)$ to be isomorphic to $X(A)$. Nevertheless, in this paper we will often find it easier to work in the ring $X(A)$ instead of $E(A)$. We note that  exact sequence 
$$0 \to I \to \Z A \to A\to 0$$
induces a map 
$ X(\Z A) \to X(A)$, with $X(I)$ in its kernel. Thus we have a continuous ring homomorphism from  
$$ E(A):= X(\Z A)/X(I) \stackrel{\pi}\longrightarrow X(A) .$$
Like $E(A)$, $X(A)$ also has a Verscheibung operator $V: X(A) \to X(A)$ given by 
$$ V(a_0,a_1,...) = p(0,a_0,a_1,...)$$
and the Teichm\"uller map $\langle \ \rangle: A \to X(A)$ given by 
$$ \langle a\rangle := (a,a^p,a^{p^2},....).$$
Clearly, the map $E(A) \xrightarrow{\pi} X(A)$ is compatible with the Verscheibung operator and Teichm\"uller map on both sides.
\end{point}

\section{Proof of Theorem \ref{thm1}}
\begin{lemma}\label{commutator-generators}Let $A$ be any associative ring. Then $\overline{[X(A),X(A)]}$ is the closed subgroup of $X(A)$ generated by 
$$ \Big\{  p^mV^n\big( [ \langle a_1 \rangle \cdots  \langle a_s \rangle, \langle b_1^{p^{n-m}} \rangle \cdots \langle b_t^{p^{n-m}}\rangle ]   \big)\ | \ m,n\in \N_0, m\leq n,  s,t\in \N, a_i,b_j \in A    \Big\} .$$
\end{lemma} 
\begin{proof}  Since $X(A)$ is the closed subgroup of $A^{\N_0}$ generated by 
$$ \big\{V^n\big( \langle a_1 \rangle \cdots \langle a_s \rangle\big) \ | \  s\in \N,\  n\in \N_0, \ a_i \in A \big\},$$
$\overline{[X(A),X(A)]}$ is the closed subgroup generated by 
$$ \big\{ \big[V^n\big( \langle a_1 \rangle \cdots \langle a_s \rangle\big), V^m\big( \langle b_1 \rangle \cdots  \langle b_t \rangle\big) \big] \ | \ s,t\in \N, \ m,n\in \N_0 , \ a_i,b_j \in A \big\}.$$
Since $$ \big[V^n\big( \langle a_1 \rangle \cdots \langle a_s \rangle\big), V^m\big( \langle b_1 \rangle \cdots  \langle b_t \rangle\big) \big] = -\big[ V^m\big( \langle b_1 \rangle \cdots  \langle b_t \rangle\big), V^n\big( \langle a_1 \rangle \cdots \langle a_s \rangle\big) \big]$$ 
$\overline{[X(A),X(A)]}$ is also the closed subgroup generated by 
$$ \big\{ \big[V^n\big( \langle a_1 \rangle \cdots \langle a_s \rangle\big), V^m\big( \langle b_1 \rangle \cdots  \langle b_t \rangle\big) \big] \ | \ s,t\in \N, \ a_i,b_j \in A, \ m,n\in \N_0 \ \text{with} \ m\leq n \big\}.$$
When $m\leq n$, a straightforward calculation shows that 
$$\big[ V^n\big( \langle a_1 \rangle \cdots \langle a_s \rangle\big), V^m\big( \langle b_1 \rangle \cdots  \langle b_t \rangle\big) \big] = p^{m} V^n\big( [\langle a_1\rangle \cdots \langle a_s \rangle, \langle b_1^{p^{n-m}} \rangle \cdots \langle b_t^{p^{n-m}}\rangle ]  \big).$$
This proves the lemma. 
\end{proof}

\begin{corollary} \label{vinduces}
$V\big( \overline{[X(A),X(A)]} \big) \subset \overline{[X(A),X(A)]}$. Thus $V$ induces a map  (to be also denoted by $V$)
$$ HH_0(X(A) ) \xrightarrow{V} HH_0(X(A)).$$
\end{corollary}

Thus like the Witt rings and like $X(A)$, $HH_0(X(A))$ also comes equipped with Verscheibung operator and a Teichm\"uller map.

\begin{lemma} \label{wagen} Let $A$ be any associative unital ring such that $\ca$ has no $p$-torsion. Then 
\begin{enumerate}
\item[(i)] for any element $(a_0,a_1,...)\in W(A)$
$$ (a_0,a_1,...) = \sum_{i=0}^\infty V^{i}\langle a_i \rangle.$$
\item[(ii)] $W(A)$ is topologically generated by the set $ \Big\{ V^n\langle a \rangle \ | \  a\in A, n\in \N_0 \Big\}$, i.e. $W(A)$ is the closure of the subgroup generated by this set. 
\end{enumerate}
 
\end{lemma}
\begin{proof} $(ii)$ is a direct consequence of $(i)$. Thus it is enough to prove $(i)$. The ghost map $W(A) \xrightarrow{\overline{\omega}} (\ca)^{\N_0}$ is a continuous group homomorphism, where addition in $(\ca)^{\N_0}$ is componentwise. Since $\ca$ has no $p$-torsion, $\overline{\omega}$ is injective. Thus, to prove $(i)$, it is enough to show $$ \overline{\omega}(a_0,a_1,...) = \sum_{i=0}^{\infty} \overline{\omega} \big( V^i(\langle a_i \rangle) \big).$$
This is checked by explicit calculation. 
\begin{align*}
\sum_{i=0}^{\infty} \overline{\omega} \big( V^i(\langle a_i \rangle) \big) & = \sum_{i=0}^{\infty} \overline{\omega}(0,0,...,0,a_i,0,...) & \text{(where $a_i$ is in the $i$-th position)} \\
& = \sum_{i=0}^{\infty}(0,0,...,0,p^ia_i,p^i a_i^{p},p^ia_i^{p^2},\cdots) & \\
& = \sum_{i=0}^{\infty} (0,0,...,0,p^ia_i,p^i a_i^{p},p^ia_i^{p^2},\cdots) & \\
 & = (a_0,a_0^p+pa_1,a_0^{p^2}+pa_1^p+p^2a_2,....) & \\
 & = \overline{\omega}(a_0,a_1,a_2,....) & 
 \end{align*}
\end{proof}

\begin{remark}
It is possible that the above lemma holds without the assumption that $\ca$ has no $p$-torsion. 
\end{remark}

Let $A={\Z}\{X,Y\}$. Then $A=\oplus_{i\geq 0} A_i$  is naturally a graded ring where $A_i$ is the free abelian group generated by words in $X,Y$ of length $i$.  The center of $A$ is $A_0=\Z$. Consider the filtration $F^{\bullet}A$ on $A$ defined by 
$$ F^nA = \oplus_{i\geq n} A_i.$$

\begin{proof}[{\bf Proof of Theorem \ref{thm1}}] $(iii)$ is a direct consequence of $(i)$ and $(ii)$. Lemma \ref{wagen} proves $(i)$. Thus to prove the theorem, it remains to show that $HH_0(E(A))$ is not topologically generated by $ \Big\{ V^n\langle a \rangle \ | \  a\in A, n\in \N_0 \Big\}$. The surjectivity of $E(A)$ to $X(A)$ implies the surjectivity of the map $HH_0(E(A)) \to HH_0(X(A))$. As this map commutes with the $V$ and $\langle \ \rangle$, it suffices to show that $HH_0(X(A))$ is not topologically generated by  $ \Big\{ V^n\langle a \rangle \ | \  a\in A, n\in \N_0 \Big\}.$ More specifically, we will show that 
the element $\langle X \rangle \langle Y \rangle \in HH_0(X(A))$ is not in the closure of the subgroup  generated by 
$ \Big\{ V^n\langle a \rangle \ | \  a\in A, n\in \N_0 \Big\}$. Suppose it is. Then, there exists elements $c_i\in A$ and $\ell_i\in \N_0$, such that in $X(A)$ we have a congruence of the form  
$$  \langle X \rangle \langle Y \rangle = \sum_i V^{\ell_i}\langle c_i \rangle \ \syn{mod} \ \overline{[X(A),X(A)]}. $$
Clearly the integers $\ell_i$ are such that the summation above converges. In particular, only finitely many of the $\ell_i$'s can be zero.  We may rewrite the above as 
$$  \langle X \rangle \langle Y \rangle - \sum_i V^{\ell_i}\langle c_i \rangle \in \ \overline{[X(A),X(A)]}.$$
By Lemma \ref{commutator-generators}, we have an equality of the form (for suitable choices of integers and elements of $A$)
$$\langle X \rangle \langle Y \rangle - \sum_{i=0}^\infty V^{\ell_i}\langle c_i \rangle = \sum_{j=0}^\infty 2^{m_j}V^{n_j} ([\langle  a_{\scr j,1} \rangle \cdots \langle a_{\scr j,s_j}\rangle, \langle b_{\scr j,1}^{2^{\scr n_j-m_j}}\rangle \cdots \langle b_{\scr j,t_j}^{2^{\scr n_j-m_j}} \rangle ]).\ \ \ \  (\star) $$
Without loss of generality, we assume that 
$$ 0 = \ell_0 = \cdots=  \ell_r < \ell_{r+1} \leq \ell_{r+2} \leq \cdots $$
$$ 0 = n_0 = \cdots=  n_\kappa < n_{\kappa+1} \leq n_{\kappa+2} \leq \cdots $$

\noindent Modulo $2$, most of the terms in $(\star)$ become zero and we have the following congruence in $\A^{\N_0}$
$$\langle X \rangle \langle Y \rangle - \sum_{i=0}^r \langle c_i \rangle = \sum_{j=0}^\kappa ([\langle  a_{\scr j,1} \rangle \cdots \langle a_{\scr j,s_j}\rangle, \langle b_{\scr j,1}\rangle \cdots \langle b_{\scr j,t_j} \rangle ] \ \ \syn{mod} \ 2 (A^{\N_0}).$$
Looking at the second component (i.e. component indexed by $1$) of the tuples we get the following congruence in the ring $A=\Z\{X,Y\}$.
$$ X^2Y^2 - \sum_{i=0}^r c_i^2 = \sum_{j=0}^\kappa [  a_{\scr j,1}^2  \cdots  a_{\scr j,s_j}^2,  b_{\scr j,1}^2 \cdots  b_{\scr j,t_j}^2  ] \ \ \syn{mod}\  2A$$
i.e. 
$$ X^2Y^2 =  \sum_{i=0}^r c_i^2 \ \ \syn{mod} \ (2A + [A,A]).$$
Observe that for any two elements $a,b\in A$ we have 
\begin{align*}
 (a+b)^2 & = a^2+b^2 + ab + ba \\
           & = a^2+b^2 \ \ \syn{mod} \ (2A + [A,A]).
\end{align*}
Thus if we let $c:= \sum_{i=0}^r c_i$, then $$c^2\equiv \sum_{i=0}^r c_i^2 \ \syn{mod} \ (2A + [A,A]).$$ 
This implies
$$ X^2Y^2 =  c^2 \ \ \syn{mod} \ (2A + [A,A]).$$
By Lemma \ref{xyc} below, this is impossible. This proves the proposition.  

\end{proof}

\begin{lemma}\label{xyc} Let $A=\Z\{X,Y\}$. Then,  $\forall \ c\in A$
$$ X^2Y^2 \neq c^2 \ \ \syn{mod} \ (2A + [A,A]+F^5A).$$ 
\end{lemma}
\begin{proof} Let $\overline{A}:= \Z/2\{X,Y\}$. As in the case of $A$, $\overline{A} =\oplus_{i\geq 0} \overline{A}_i$ 
where $\overline{A}_i$ is $\Z/2$ span of words in $X,Y$ of length precisely $i$. We let $F^n\overline{A}:=\oplus_{i\geq n} A_i.$
In order to prove the lemma it is enough to arrive at a contradiction by assuming the existence of an element $c\in \overline{A}$, such that  
$$ \hspace{3cm} X^2Y^2 = c^2 \ \ \syn{mod} \ ( [\overline{A},\overline{A}]+F^5\overline{A}). \ \hspace{3cm} \ \ \ \ {(\star\star)}$$ 
Without loss of generality we may assume that $c=a_0+a_1+a_2+a_3+a_4$ where $a_i\in \overline{A}_i$ are its homogeneous components. But then 
\begin{align*}
 c^2 & = a_0^2+a_1^2 + a_2^2 + a_3^2 + a_4^2 \ \ \syn{mod} \  ([\overline{A},\overline{A}]+ F^5 \overline{A})\\
       & = a_0^2 + a_1^2 + a_2^2 \ \ \syn{mod}  \ ([\overline{A},\overline{A}]+ F^5 \overline{A}).
\end{align*}
Thus, $(\star \star)$ continues to hold with   $c$ replaced by $a_0+a_1 + a_2$. Hence without loss of generality we may assume $ c= a_0+a_1 + a_2$. \\ 

\noindent \underline{Step 1}: Let $\tilde{A}:= \Z/2[X,Y]$ be the commutative polynomial ring in variables $X,Y$. This is quotient of the non-commutative polynomial ring $\overline{A}=\Z/2\{X,Y\}$, i.e. we have a natural surjective map 
$$ \overline{A} \xrightarrow{\phi} \tilde{A}$$
with $[\overline{A},\overline{A}] \subset \Ker(\phi)$. The homomorphism $\phi$ preserves grading. 
We let $F^n\tilde{A}:= \oplus_{i\geq n} \tilde{A}_i$. Applying $\phi$ to the congruence $(\star \star)$ above, we get the following congruence in $\tilde{A}$  
\begin{align*}
 X^2Y^2 & = \phi(c)^2 \ \ \syn{mod} \ F^5\tilde{A}\\
               & = \phi(a_0)^2 + \phi(a_1)^2 + \phi(a_2^2) \ \ \syn{mod} \ F^5\tilde{A}
\end{align*}
Note that $\phi(a_i) \in \tilde{A}_i$. Thus both sides of the congruence belong to the subgroup $\displaystyle{\oplus_{i=0}^4 \tilde{A}_i}$, which has trivial intersection with $F^5\tilde{A}$. Therefore the congruence gives us an  equality in $\tilde{A}$
$$ X^2Y^2 = \phi(a_0)^2 + \phi(a_1)^2 + \phi(a_2^2).$$
Further, since $X^2Y^2 \in \tilde{A}_4$, the only possibility is $\phi(a_0)=\phi(a_1)=0$ and $\phi(a_2)=XY$. 
Note that $\phi_{| \overline{A}_i}: \overline{A}_i \longrightarrow \tilde{A}_i$ is an isomorphism for $i=0,1$. Thus $$a_0=a_1=0.$$
Moreover, 
$$\Ker\big(\overline{A}_2\xrightarrow{\phi_{|\overline{A}_2}} \tilde{A}_2\big)=\Z/2\ \text{span of }  XY-YX.$$ 
Thus $\phi(c)=\phi(a_2)=XY$ implies that 
$$ c = XY + \epsilon [X,Y] \ \ \ \text{for some }  \epsilon \in \Z/2.$$ 
This implies  
$$c^2    = (XY)^2 + \epsilon^2( XY-YX)^2 \ \ \syn{mod} \ [\overline{A},\overline{A}].$$

\noindent \underline{Step 2}:  
\begin{align*}
( XY-YX)^2  &  =  (XY)^2 + (YX)^2  \ \ \syn{mod} \ [\overline{A},\overline{A}] \\
          & = [X,YXY] \ \ \syn{mod} \ [\overline{A},\overline{A}] \\
        & = 0 \ \ \syn{mod} \ [\overline{A},\overline{A}]
\end{align*}
Thus $$ c^2 = (XY)^2 \ \ \syn{mod} \ [\overline{A},\overline{A}].$$
Therefore, without loss of generality, we may assume $c=XY$. \\

\noindent \underline{ Step 3}: Now $(\star \star) $ gives us 
$$ X^2Y^2 - XYXY = u + v $$
for some $u\in [\overline{A},\overline{A}]$ and $v\in F^5\overline{A}$. 
Write $u=\sum_i u_i$ with $u_i\in \overline{A}_i \ \forall \ i$. Comparing the degree $4$ homogeneous components of the above equation, we get 
$$ X^2Y^2-XYXY = u_4.$$
Note that $[\overline{A},\overline{A}]$ is graded subgroup of $\overline{A}$. Thus $u_4\in [\overline{A},\overline{A}]$. 
Let $S \subset \overline{A}_4$ be a subset defined by  
$$ S:= \big\{ \text{all words in } X,Y \text{ of length } 4, \text{except } XYXY \text{ and } YXYX  \big\} \Union \big\{ XYXY-YXYX \big\}.$$
Let $K\subset \overline{A}_4$ be the $\Z/2$-subspace spanned by $S$. Clearly 
$$ X^2Y^2-XYXY \notin K.$$
Thus to finish the proof, it suffices to show $u_4\in K$. This will be done in the next step. \\

\noindent \underline{ Step 4}: $u_4 \in [\overline{A},\overline{A}]\intersection \overline{A}_4$. It is elementary to check that $[\overline{A},\overline{A}]\intersection \overline{A}_4$ is generated by the set 
$$ \big\{ a_ib_j - b_ja_i \ | \  (i,j)= (1,3) \text{ or } (2,2) , \text{ with } \ a_i, b_j \text{ are words of length } i,j \text{ respectively}  \big\}.$$
It therefore suffices to show that all elements of the form $a_ib_j-b_ja_i$ of the above set belong to $K$. 
The possibilities for $a_i$ and $b_j$ appearing in the above set are listed in the table below: \\
\begin{center}
\begin{tabular}{ |c| c | }
\hline 
$a_1$ & $X,Y$  \\
\hline 
$b_3$ & $X^3,Y^3,X^2Y,XY^2,YX^2,Y^2X, XYX,YXY$ \\
\hline 
$a_2$ or $b_2$ & $X^2,Y^2,XY,YX$ \\
\hline
\end{tabular}
\end{center}
For all possible combinations of $a_1$ and $b_3$, one explicitly checks that $a_1b_3-b_3a_1$ always belongs to $K$. Similarly for all possible combinations of $a_2$ and $b_2$ from the above table, one checks that $a_2b_2-b_2a_2$ always belongs to $K$. This finishes the proof. 

\end{proof}


\section{Proof of Theorem \ref{thm2}}

The goal of this section is to prove Theorem \ref{thm2}. Define a set map $\Omega: A^{\N_0} \to A^{\N_0}$  by 
$$\Omega(\underline{a}) = (\omega_0(\underline{a}), \omega_1(\underline{a}), ...)$$ 
where $\omega_n(\underline{a})  = a_0^{p^n}+pa_1^{p^{n-1}}+\cdots+p^na_n$ are the Witt polynomials.
Recall that $X(A)\subset \A^{\N_0}$ is the closed subgroup generated by elements of the form $V^i\big(\langle a_1\rangle \cdots \langle a_r \rangle\big)$ where $V$ and $\langle \ \rangle$ are as in the construction by Cuntz and Deninger summarized in the previous section.

\begin{lemma} \label{xqvec} 
Image of $\A^{\N_0}\xrightarrow{\Omega} \A^{\N_0}$ is contained in $X(A)$. In fact, we have the following equality in $X(A)$:
$$  \Omega(a_0,a_1,...)= \sum_{i=0}^\infty V^i \langle a_i \rangle  \ \ \ \forall \  
(a_0,a_1,...) \in A^{\N_0}.$$
\end{lemma}
\begin{proof}
\begin{align*}
\Omega(a_0,a_1,...) & = (\omega_0(\underline{a}), \omega_1(\underline{a}), ...) \\
	  & = (a_0, a_0^p+pa_1, a_0^{p^2}+pa_1^p+p^2a_2, ....)\\
	  &= (a_0,a_0^p,a_0^{p^2},...) + p(0,a_1,a_1^p,a_1^{p^2},...) + p^2 (0,0,a_2,a_2^p,a_2^{p^2},...)+...\\
	  &= \sum_{i=0}^\infty V^i \langle a_i \rangle
\end{align*}
\end{proof}

\noindent Thus, by the above lemma, $\Omega$  induces a map from $\A^{\N_0}\to X(A)$ which will also be denoted by $\Omega$.

\begin{lemma} \label{pin} Let $A$ be such that $\ca$ is p-torsion free. Suppose there exists a continuous map $\Phi: W(A) \to HH_0(X(A))$ which commutes with $\langle \ \rangle$ and  $V$. 
Then the following diagram commutes 
$$\xymatrix{
& & X(A) \ar@{->>}[d] \\
& &  HH_0(X(A)) \\
A^{\N_0} \ar[rr]^{q} \ar[rruu]^{\Omega} & &  W(A)\ar[u]_{\Phi}
}$$ 
In other words, $\Phi$ is necessarily given by 
$$ \Phi(q(\underline{a})) = \Omega (\underline{a})  \ \syn{mod}\  \overline{[X(A),X(A)]} .$$
\end{lemma}

\begin{proof} For $\underline a = (a_0,a_1,...)\in \A^{\N_0}$ we have 
\begin{align*}
\Phi \big( q(\underline{a})\big) & =  \Phi\Big( \sum_{i=0}^\infty V^i\langle a_i \rangle \Big) &  (\text{by lemma \ref{wagen}(i) } ) \\
                          & =   \sum_{i=0}^\infty \Phi(V^i\langle a_i \rangle ) & (\text{by continuity of $\Phi$}) \\
                          & = \Omega (\underline{a})  \ \syn{mod}\  \overline{[X(A),X(A)]} & (\text{by Lemma \ref{xqvec} and since $\Phi$ is continuous and preserves $V$, $\langle \ \rangle$})
\end{align*}
This proves the lemma. 
\end{proof}


To state the next result we fix the following notation. Let $ X(A) \xrightarrow{\ \gamma \ } \big(\ca\big)^{\N_0}$ denote the composition 
$$   X(A) \inj A^{\N_0} \to \Big(\ca\Big)^{\N_0}.$$
Let $ E(A) \xrightarrow{\ \eta\ } \big(\ca \big)^{\N_0}$ denote the composition 
$$ E(A) \xrightarrow{\ \pi\ } X(A) \xrightarrow{\ \gamma\ } \Big(\ca\Big)^{\N_0}$$
where $\pi$ is the surjection induced by $X(\Z A) \to X(A)$ (see subsection \ref{eaxa}). Since $\eta$ and $\gamma$  are ring homomorphisms, and the target is a commutative ring, we have induced maps 
$$ HH_0(X(A)) \xrightarrow{ \ \overline{\gamma} \ } \Big(\ca\Big)^{\N_0},$$
$$ HH_0(E(A)) \xrightarrow{ \ \overline{\eta} \ } \Big(\ca\Big)^{\N_0}.$$
The maps  $\eta$ and $\overline{\eta}$ are analogous to the ghost map 
$$ W(A) \xrightarrow{\ \overline{\omega}\ } \Big( \ca \Big)^{\N_0}.$$
In the case when $\ca$ is is $p$-torsion free, $\overline{\omega}$ is always injective (see \cite[page 56]{h2}). However as the following theorem  shows, this is not the case for  $\overline{\eta}$. 

\begin{theorem}\label{notinjective} Let $A = \Z\{X,Y\}$, and $p=2$.
The map $$HH_0(E(A)) \xrightarrow{\ \overline{\eta} \ } \Big(\ca\Big)^{\N_0}$$ is not injective. 
\end{theorem}

\noindent To prove this theorem we need the following lemma and a construction from \cite{h1} (see \ref{mapr}). 

\begin{lemma}\label{thelemma} Let $A=\Z\{X,Y\}$ and $p=2$. 
Let $A_4^0 \subset A_4$ be the free subgroup generated by  all words of length $4$ in $X,Y$ except the words $XYXY$ and $YXYX$. Let $H:= A_4^0 + F^5A + 2A$. Then $$(x_0,x_1,x_2,...)\in \overline{[X(A),X(A)]}\implies x_1  = 0 \ \syn{mod} \ H.$$
\end{lemma}
\begin{proof} 
By Lemma \ref{commutator-generators}, $\overline{[X(A),X(A)]}$ is the closed subgroup generated by the following elements of $X(A)$:
$$ \Big\{  2^mV^n\big( [ \langle a_1 \rangle \cdots  \langle a_s \rangle, \langle b_1^{2^{n-m}} \rangle \cdots \langle b_t^{2^{n-m}}\rangle ]   \big)\ | \  0\leq m\leq n\leq 1,  s,t\in \N, a_i,b_j \in A   \Big\} $$
where recall that for $a\in A$, $\langle a \rangle= (a,a^2,a^{2^2},...)\in X(A)$.  
As $A_0$ is in the center of $A$ we have 
$$\langle c \rangle \langle d \rangle = \langle cd \rangle  \ \forall \ c\in A_0 \ \text{and } d\in A. $$ 
Thus, without loss of generality, we may restrict the $a_i, b_j$ in the generating set above to elements of $A\backslash A_0$. In other words  $\overline{[X(A),X(A)]}$ is the closed subgroup generated by the following elements of $X(A)$:
$$ \Big\{  2^mV^n\big( [ \langle a_1 \rangle \cdots  \langle a_s \rangle, \langle b_1^{2^{n-m}} \rangle \cdots \langle b_t^{2^{n-m}}\rangle ]   \big)\ | \  0\leq m\leq n\leq 1,  s,t\in \N, a_i,b_j \in A  \backslash A_0 \Big\} $$

Thus it is enough to prove the statement of the lemma when $(x_0,x_1,..)$ is an element of the generating set, i.e.  $$(x_0,x_1,....)= 2^mV^n\big( [ \langle a_1 \rangle \cdots  \langle a_s \rangle, \langle b_1^{2^{n-m}} \rangle \cdots \langle b_t^{2^{n-m}}\rangle ]   \big).$$
Clearly if $m> 0$ or if $n> 0$, the RHS is divisible by $2$ and since $2A\subset H$, we have $x_1\in H$. It thus remains to show that for  
 $$(x_0,x_1,...)=  [ \langle a_1 \rangle \cdots  \langle a_s \rangle, \langle b_1 \rangle \cdots \langle b_t\rangle ] ,  $$
where $s,t\geq 1$, $x_1\in H$. Equivalently, we need to show the following:
$$ \big[a_1^2\cdots a_s^2, b_1^2\cdots b_t^2\big] = 0 \ \ \syn{mod} \  H \ \ \ \forall \ a_i, b_i \in A. $$
To see this, we first observe that $H$ is a two sided ideal of $A$.  The above statement is equivalent to showing that  images of the elements $(a_1^2\cdots a_s^2)$ and $(b_1^2\cdots b_t^2)$ in the quotient ring $A/H$ commute. For this it is enough to show that for all $1\leq i\leq s$ and $1\leq j\leq t$, the images of the elements $a_i^2$ and $b_j^2$ in $A/H$ commute. In other words, we are reduced to showing  that for  elements $a,b\in A$, 
$$ a^2b^2=b^2a^2 \ \ \ \syn{mod} \ H.$$
Write 
\begin{align*}
a &=\alpha_0+\alpha_1+\alpha_2 \\
b &=\beta_0+\beta_1+\beta_2
\end{align*} where $\alpha_0,\beta_0\in A_0$, $\alpha_1,\beta_1\in A_1$ and $\alpha_2,\beta_2\in F^2A$.
Then, since $2A\subset H$, we have  
\begin{align*}
a^2 &  = \alpha_0^2 + \alpha_1^2  + \alpha_2^2+ \alpha_1\alpha_2 +\alpha_2\alpha_1 \ \ \syn{mod} \ H \\
b^2 & = \beta_0^2 + \beta_1^2  + \alpha_2^2+ \beta_1\beta_2 +\beta_2\beta_1 \ \ \syn{mod} \ H 
\end{align*}
We now explicitly calculate $a^2b^2-b^2a^2$ modulo $H$, while keeping in mind that $\alpha_0,\beta_0$ are in the center of $A$.  
\begin{align*}
a^2b^2-b^2a^2 & = \big(\alpha_0^2 + \alpha_1^2  + \alpha_2^2+ \alpha_1\alpha_2 +\alpha_2\alpha_1 \big) \big(\beta_0^2 + \beta_1^2  + \beta_2^2+ \beta_1\beta_2 +\beta_2\beta_1  \big) \\ 
                         & \ \ \ -  \big( \beta_0^2 + \beta_1^2  + \beta_2^2+ \beta_1\beta_2 +\beta_2\beta_1 \big) \big(\alpha_0^2 + \alpha_1^2  + \alpha_2^2+ \alpha_1\alpha_2 +\alpha_2\alpha_1 \big) \ \ \syn{mod} \ H \\
                         & = \alpha_1^2\beta_1^2-\beta_1^2\alpha_1^2 + ({\rm deg } \geq 5 \ {\rm terms}) \ \ \syn{mod} \ H \\
                         & = \alpha_1^2\beta_1^2-\beta_1^2\alpha_1^2  \ \ \syn{mod} \ H 
\end{align*}
Thus to finish the proof, it suffices to show that $\alpha_1^2\beta_1^2-\beta_1^2\alpha_1^2 \in A_4^0$. Since $\alpha_1,\beta_1\in A_1$, there exists integers $m_i,n_i$ for $i=1,2$ such that 
\begin{align*}
\alpha_1 & = m_1X + n_1Y \\
\beta_1 & = m_2X + n_2Y
\end{align*}
It is now elementary to check that in the expression of $\alpha_1^2\beta_1^2-\beta_1^2\alpha_1^2 $, which is a linear combination of words of length $4$ in $X$ and $Y$, the coefficients of the words $XYXY$ and $YXYX$ are zero. By definition of $A_4^0$, this shows that $\alpha_1^2\beta_1^2-\beta_1^2\alpha_1^2 \in A_4^0$. 
\end{proof}


\begin{point} {\bf Hesselholt's construction of the map $[A,A]^{\N_0}\longrightarrow A^{\N_0}$}: \label{mapr}
In this subsection we quickly recall the construction of a map from $[A,A]^{\N_0}\longrightarrow A^{\N_0}$ by Hesselholt (see \cite[Page 114]{h1}). Everything recalled here is taken directly from \cite{h1} and nothing is new, except that instead of the non-unital polynomial ring $\tilde{\Z}\{X,Y\}$ we deal with the unital polynomial ring 
$\Z\{X,Y\}$.  Let $A=\Z\{X,Y\}$. Fix a prime $p$.  An element $x\in A$ can be uniquely written as 
$$ x= \sum_{{w}} n_{{w}} {w} \ \ \ , \ \ \  n_{{w}}\in \Z \ \text{ with almost all zero}$$
where the summation is over words ${w}$ in $X,Y$ (including the empty word). 
Define two words to be equivalent, if one can be obtained by a cyclic permutation of the other. An equivalence class of words will be called a circular word. One can show that $A/[A,A]$ is a free abelian group generated by circular words. Define a additive section of the quotient map $A\twoheadrightarrow A/[A,A]$  
$$ \sigma_0: \frac{A}{[A,A]} \longrightarrow A$$ 
which takes a circular word to the unique word in its equivalence class which is least for the lexicographic order.  Define an additive group homomorphism $A\xrightarrow{\phi}A$ by 
$$ \phi\big(  \sum_{w} n_{w} w\big) :=  \sum_{w} n_{w} w^p.$$
As observed in \cite{h1}, $\phi$ satisfies the following. 
\begin{lemma}\label{phi} $\phi([A,A])\subset [A,A]$ and for all $x\in A$, 
$ x^{p^n} = \phi(x^{p^{n-1}}) \ \syn{mod} \ \big(p^nA + [A,A] \big).$
\end{lemma}
\vspace{2mm}
\noindent We now define a map 
$$ \sR : [A,A]^{\N_0} \longrightarrow A^{\N_0}$$
 $$\sR(\epsilon_0,\epsilon_1,...) := (r_0,r_1,...)$$ 
 where $r_0=\epsilon_0$ and for $n\geq 1$,  $r_n$'s are recursively defined as:
 $$ r_n = \epsilon_n - \sigma_0 \Big( p^{-n}\big( \omega_n(r_0,r_1,...,r_{n-1},0)-\phi(\omega_{n-1}(r_0,...,r_{n-1})) \big) \Big).$$
\end{point}
In the above formula, we note that by $\big( \omega_n(r_0,r_1,...,r_{n-1},0)-\phi(\omega_{n-1}(r_0,...,r_{n-1})) \big)$ we actually mean its image in $A/[A,A]$, which is divisible by $p^n$ by Lemma \ref{phi}. Moreover, since $A/[A,A]$ has no $p$-torsion, $p^{-n}\big( \omega_n(r_0,r_1,...,r_{n-1},0)-\phi(\omega_{n-1}(r_0,...,r_{n-1})) \big)$ is a  well defined element of $A/[A,A]$. \\ 

\noindent Finally we recall the part of \cite[Lemma 1.3.7]{h1} which remains true. 
\begin{lemma}\label{omegar0}
For any $\underline{\epsilon}\in [A,A]^{\N_0}$, $\sR (\underline{\epsilon})$ maps to zero via the ghost map 
$$ A^{\N_0} \xrightarrow{ \ \omega \ } \Big( \ca \Big)^{\N_0}.$$ 
\end{lemma}

\begin{proof} [{\bf Proof of Theorem \ref{notinjective}}]
As the map $HH_0(E(A)) \longrightarrow HH_0(X(A))$ is surjective, it is enough to show that the map 
$$ HH_0(X(A)) \xrightarrow{\ \overline{\gamma} \ } \Big(\ca \Big)^{\N_0}$$ is not injective.  Consider the following diagram
$$\xymatrix{
 & X(A) \ar[r]  & HH_0(X(A)) \ar[d]^{\overline{\gamma}}\\
A^{\N_0} \ar[ru]^{\Omega} \ar[rr]^{\omega} &  & \Big(\ca\Big)^{\N_0}
}$$ 
In order to prove the theorem it suffices to find an element $r=(r_0,r_1,...)\in A^{\N_0}$ such that $\Omega(r)\notin \overline{[X(A),X(A)]}$ and $\omega(r)=0$. Let 
$$ r:=(r_0,r_1,r_2,...):= \sR(XY-YX,0,0,...) \in A^{\N_0}.$$ 
By Lemma \ref{omegar0}, $\omega(r)=0$. It remains to show that $\Omega(r)\notin \overline{[X(A),X(A)]}$. By Lemma \ref{thelemma}, it suffices to show that the $1$-th component of $\Omega(r)$ (i.e. $\omega_1(r)$) is not in $A_4^0+F^5A+2A$. 
\begin{align*}
\omega_1(r)   & = {r_0}^2+2r_1\\
                & = (\rm {XY-YX})^2+2(-{\rm XYXY+XXYY})\\
                & =  -{\rm XYXY+YXYX-XYYX-YXXY+2XXYY}\\
                & = -{\rm XYXY+YXYX} \ \syn{ mod} \ A_4^0+F^5A+2A. 
\end{align*}
As a direct consequence of the definition of $A_4^0$ and $F^5A$, one thus sees that $\omega_1(r)\notin A_4^0+F^5A+2A$. This finishes the proof.

\end{proof}

\begin{proof}[{\bf Proof of Theorem \ref{thm2}}]
Suppose there exists a continuous group homomorphism $W(A) \to HH_0(E(A))$. Composing with the natural homomorphism $HH_0(E(A)) \to HH_0(X(A))$ we get a map $$\Phi: W(A) \to HH_0(X(A)).$$ 
Let us denote the quotient map $X(A) \to HH_0(X(A))$ by $\pi$. By lemma \ref{pin}, the following diagram must commute:
$$\xymatrix{
 X(A)\ar@{->>}[r]^-{\pi} & HH_0(X(A)) \\
A^{\N_0} \ar[r]^q \ar[u]^{\Omega}  & W(A) \ar[u]^-{\Phi}
}$$

\noindent Using surjectivity of $q$ and the fact that both $\Omega$ and $\overline{\omega}\circ q$ are given by the same Witt polynomials, one checks that the above commutative square can be extended to the following commutative diagram.

$$\xymatrix{
   X(A)\ar@{->>}[r]^-{\pi} & HH_0(X(A)) \ar[r]^-{\overline{\gamma}}             & \Big( \ca\Big)^{\N_0}\ar@{=}[d]\\
A^{\N_0}  \ar[r]^q \ar[u]^{\Omega}  & W(A) \ar[u]^-{\Phi}\ar[r]^{\overline{\omega}} & \Big(\ca\Big)^{\N_0}
}$$
However, such a commutative diagram is not possible as there exists an element $r\in A^{\N_0}$ (see proof of Theorem \ref{notinjective}) such that 
$ \overline{\gamma}\circ \pi \circ \Omega(r) \neq 0$ but $\overline{\omega}\circ q (r)=0$. 
\end{proof}


\end{document}